\documentclass[10pt]{amsart}

\IfFileExists{srcltx.sty}{\usepackage{srcltx}}

\numberwithin{equation}{section}

\usepackage[latin1]{inputenc}
\usepackage{xspace,amssymb,amsfonts,euscript}
\usepackage{amsthm,amsmath}
\usepackage{palatino}
\usepackage{euscript}
\input xy \xyoption {all}

\input{tableau}
\newcommand{\rc}{\gentabbox{1}{1}{0.4pt}{2pt}{0.4pt}{0.4pt}{{}}}

\newcommand{\row}[1]{\hbox{$\tabstyle #1$}}
\newcommand{\q}{\genblankbox{1}{1}\relax}

\RequirePackage{color}
\definecolor{myred}{rgb}{0.75,0,0}
\definecolor{mygreen}{rgb}{0,0.5,0}
\definecolor{myblue}{rgb}{0,0,0.65}

\RequirePackage{ifpdf}
\ifpdf
  \IfFileExists{pdfsync.sty}{\RequirePackage{pdfsync}}{}
  \RequirePackage[pdftex,
   colorlinks = true,
   urlcolor = myblue, 
   citecolor = mygreen, 
   linkcolor = myred, 
   pagebackref,
   bookmarksopen=true]{hyperref}
\else
  \RequirePackage[hypertex]{hyperref}
\fi

\RequirePackage{ae, aecompl, aeguill} 

    \def\FM{{\mathbb{F}}}
  \def\gg{{\mathfrak g}}  \def\GM{{\mathbb{G}}}

    \def\LM{{\mathbb{L}}}

    \def\OM{{\mathbb{O}}}
  \def\pg{{\mathfrak p}}  \def\PM{{\mathbb{P}}}
  \def\qg{{\mathfrak q}}

  \def\ug{{\mathfrak u}}

    \def\ZM{{\mathbb{Z}}}

    \def\EC{{\mathcal{E}}}
    \def\FC{{\mathcal{F}}}

    \def\NC{{\mathcal{N}}}
    \def\OC{{\mathcal{O}}}
    
    \def\QC{{\mathcal{Q}}}

\newcommand{\nc}{\newcommand} \newcommand{\renc}{\renewcommand}

\newcommand{\rdots}{\mathinner{ \mkern1mu\raise1pt\hbox{.}
    \mkern2mu\raise4pt\hbox{.}
    \mkern2mu\raise7pt\vbox{\kern7pt\hbox{.}}\mkern1mu}}

\def\wt{\widetilde}
\def\ov{\overline}

\def\to{\rightarrow}

\def\longto{\longrightarrow}

\nc{\triright}{\stackrel{[1]}{\to}}
\nc{\longtriright}{\stackrel{[1]}{\longto}}

\nc{\Br}{\mathcal{B}}
\nc{\HotRR}{{}_R\mathcal{K}_R}
\nc{\HotR}{\mathcal{K}_R}
\nc{\excise}[1]{}
\nc{\defect}{\text{df}}
\nc{\h}[1]{\underline{H}_{#1}}

\nc{\Ga}{\mathbb{G}_a} 
\nc{\Gm}{\mathbb{G}_m} 

\nc{\Perv}{{\mathbf{P}}}

\nc{\IH}{{\mathrm{IH}}}

\nc{\ic}{\mathbf{IC}}

\nc{\gl}{{\mathfrak{gl}}}
\renc{\sl}{{\mathfrak{sl}}}
\renc{\sp}{{\mathfrak{sp}}}

\nc{\HBM}{H^{BM}}

 \DeclareMathOperator{\Hom}{Hom}
 
\DeclareMathOperator{\End}{End} 

\newtheorem{thm}{Theorem}[section]
\newtheorem{lem}[thm]{Lemma}

\newtheorem{prop}[thm]{Proposition}
\newtheorem{cor}[thm]{Corollary}

\theoremstyle{definition}

\theoremstyle{remark}

\def\uk{\underline{k}}

\begin{document}

\begin{abstract}
We construct affine pavings of Springer-type fibers over the enhanced nilpotent cone.  This resolves a question of Achar-Henderson and implies the existence of perverse parity sheaves on the enhanced nilpotent cone.
\end{abstract}

\title{Affine pavings and the enhanced nilpotent cone}

\author{Carl Mautner} \address{ Department of Mathematics, University of California, Riverside, CA 92521, USA}
\email{mautner@math.ucr.edu}

\maketitle


\section{Notation and Results}


Let $\mathbb F$ be an algebraically closed field of arbitrary characteristic and let $V$ be an $n$-dimensional $\mathbb F$-vector space.  Let $G = GL(V)$ and $\gg = Lie(G)$ be its Lie algebra with nilpotent cone $\NC \subset \gg$.  The $G$-variety $V \times \NC$ is known as the enhanced nilpotent cone.  As shown independently by~\cite{AH} and~\cite{Tr}, the $G$-orbits in $V \times \NC$ are in bijection with the set $\QC_n$ of bipartitions of $n$ (meaning ordered pairs of partitions $(\mu;\nu)$ such that $|\mu|+|\nu|=n$).  The closure of each orbit $\OC_{\mu;\nu}$ has a semismall resolution of singularities $\pi_{\mu;\nu} : \wt{\FC_{\mu;\nu}} \to V \times \NC$ (whose construction we recall below).  The aim of this paper is to construct affine pavings of the fibers of these resolutions.  This claim appeared in~\cite{AH}, but was then retracted in~\cite{AHcor}, where it is posed as an open problem.  Our construction is a variant of the method introduced in \cite{DLP} to construct affine pavings of Springer fibers for classical groups.

To describe the resolutions $\pi_{\mu;\nu}$, recall that \cite{AH} associate a `back-to-back union' diagram to $(\mu;\nu)$, whose $i$-th row contains $\mu_i + \nu_i$ boxes and $(\mu_1-i)$-th column has $\mu^t_{i+1}$ boxes for $i \geq 0$ and $(\mu_1+i)$-th column has $\nu^t_{i}$ boxes for $i>0$.  For example, the diagram associated to $((3,1,1);(3,2)) \in \QC_{10}$ is represented as:
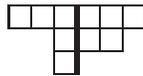
\[\begin{tableau}
\row{\c\c\rc\c\c \c}
\row{\q\q\rc\c\c}
\row{\q\q\rc}
\end{tableau}
\]

Let $\FC_{\mu;\nu}$ be the variety of partial flags 
\[0 = W_0 \subset W_1 \subset \dots \subset W_{\mu_1+\nu_1} = V,\]
where $W_i$ has dimension equal to the number of boxes in or to the
left of the $i$-th column in the diagram of
$(\mu;\nu)$.  So in the example above, $\mu_1+\nu_1 = 6$ and the dimensions of the subspaces are: 1, 2, 5, 7, 9 and 10.

We will consider, more generally, for any sequence $\rho$, $0=r_0 < r_1 < \dots < r_m = n$, the variety $\FC_\rho$ of partial flags 
\[0 = W_0 \subset W_1 \subset \dots \subset W_m = V,\]
where the dimension of $W_i$ is $r_i$.

Recall the resolution of $\ov{\OC_{\mu;\nu}}$ defined in \cite{AH} via the space 
\[ \widetilde{\FC_{\mu;\nu}} := \{(v,x,(W_i)) \in V\times \NC \times
\FC_{\mu;\nu} | v \in W_{\mu_1}, x(W_i) \subset W_{i-1} \},\]
and the projection $\pi_{\mu;\nu} : \wt{\FC_{\mu,\nu}} \to
V \times \NC$ to the first two coordinates. By \cite[Thm. 4.5]{AH}, $\pi_{\mu;\nu}$ is a semismall resolution of $\ov{\OC_{\mu;\nu}}$.

More generally, for any $j\in \ZM$ such that $0 \leq j \leq m$, let $\widetilde{\FC_{\rho,j}}$ be defined as
\[ \widetilde{\FC_{\rho,j}} := \{(v,x,(W_i)) \in V\times \NC \times
\FC_{\rho} | v \in W_{j}, x(W_i) \subset W_{i-1} \},\]
and let $\pi_{\rho,j} : \widetilde{\FC_{\rho,j}} \to V \times \NC$ denote the projection.

Our main result is the following:

\begin{thm}
\label{main}
For any $(v,x) \in V \times \NC$, the fiber $\pi_{\rho,j}^{-1}(v,x)$ has an affine paving.  In particular, the fiber $\pi_{\mu;\nu}^{-1}(v,x)$ admits an affine paving.
\end{thm}

As a simple corollary, we observe that this implies the existence of perverse parity sheaves\footnote{We consider here only the constant pariversity.} on the enhanced nilpotent cone.  For simplicity, we assume for the rest of the introduction that $\mathbb F=\mathbb C$ the field of complex numbers.  Let $k$ be a complete local principal ideal domain.
Let $D_{G}(V \times \NC;k)$ denote the $G$-equivariant constructible derived category of $k$-sheaves.

\begin{cor}
For each $G$-orbit $\OC_{\mu;\nu}$, there exists up to isomorphism one parity sheaf $\EC_{\mu;\nu} \in D_G(V \times \NC;k)$ with support $\overline{\OC_{\mu;\nu}}$, and it is perverse.
\end{cor}

\begin{proof}
First note that there are finitely many $G$-orbits in $V \times \NC$ and for any $(v,x) \in V\times \NC$ the stabilizer is connected \cite[Prop. 2.8(7)]{AH} and has reductive quotient isomorphic to a product of general linear groups \cite[Thm. 2.12]{Sun}.  It follows that the orbits are equivariantly simply connected and have equivariant cohomology concentrated in even degrees.  Thus, as a $G$-variety, the enhanced nilpotent cone satisfies the parity conditions of \cite{JMW2}, which implies the uniqueness statement. 

For the existence of $\EC_{\mu;\nu}$, note that the resolution $\pi_{\mu;\nu}$ is semismall, so the push-forward sheaf $(\pi_{\mu;\nu})_* \uk_{\widetilde{\FC_{\mu;\nu}}} [\dim \OC_{\mu;\nu}]$ is perverse and Theorem~\ref{main} implies that it is also a parity complex.  It follows that the push-forward sheaf, which has support $\overline{\OC_{\mu;\nu}}$, has a perverse indecomposable parity complex $\EC_{\mu;\nu}$ with support $\overline{\OC_{\mu;\nu}}$ as a direct summand.
\end{proof}

\subsection{Acknowledgements}  The author is grateful to Gwyn Bellamy for a useful remark and to Anthony Henderson for his interest and comments on an early draft. This work was done in part while the author was visiting the Max Planck Institut f\"ur Mathematik in Bonn and he would like to thank MPIM for excellent working conditions.

\section{Construction of Affine Paving}

\subsection{}
\label{subsec-alphapart}

As defined and shown in \cite{AH}, we may pick a
normal basis for $(v,x)$.  In this basis, each basis vector of $V$ corresponds to a box of the back-to-back union diagram for $(\alpha;\beta)$.
We denote by $v_{i,j}$ the basis vector corresponding to the $j$-th box in the $i$-th row.  In this basis the action of $x$ is given by $xv_{i,j} = v_{i,j-1}$ (or $0$ if $j=1$), and the vector $v$ expressed as $v = \sum_{i = 1}^{\alpha^t_1} v_{i,\alpha_i}$.  For example, for $((3,1,1);(3,2)) \in \QC_{10}$ we have basis vectors:
\[
\begin{array}{cccccc}
v_{11}&v_{12}&v_{13}&v_{14}&v_{15}&v_{16}\\
&&v_{21}&v_{22}&v_{23}&\\
&&v_{31}&&&\\
\end{array}
\]
and $v = v_{13}+v_{21}+v_{31}$.

We grade $V$ by giving the basis vector $v_{i,j}$ grading $\alpha_i -j$.  Let
$V(i)$ denote the $i$-th graded part.  This induces a grading on $\gg
= \Hom(V,V)$.  Let $\gg(i)$ denote the $i$-th graded part of $\gg$
(i.e., $\oplus_j \Hom(V(j),V(j+i))$).  Let $V^{\geq 0} = \oplus_{i \geq 0}
V(i)$ be the non-negatively graded part of $V$.  (In the notation of \cite{AH}, $V^{\geq 0} = E^x v$. See Proposition 2.8(5) of \textsl{loc. cit.})

Note that $v \in V(0)$ and $x \in \gg(1)$.

Consider the parabolic subalgebra $\pg = \oplus_{i\geq 0} \gg(i)$, its Levi subalgebra $\gg(0) = \oplus_i \End(V(i))$ and unipotent radical $\ug_P = \oplus_{i > 0} \gg(i)$.  Let $G_0 = \prod_i GL(V(i))$ and $P$ be the
corresponding Levi and parabolic subgroups of $G$.  (In \cite[following Thm. 4.1]{AH}, $P$ is denoted $P^{(v,x)}$.) 

Let $\lambda:\GM_m \to G$ denote a cocharacter inducing
this Levi decomposition.  

\begin{lem}
The $P$-orbit of $(v,x)$ in $V^+ \times \ug_P$ is dense.
\end{lem}

\begin{proof}
This is Lemma 4.2 of \cite{AH}.  In \textit{loc.~cit.}, $\FM$ is assumed to be the field of complex numbers, but the same proof applies more generally.
%
%
%
%
%
\end{proof}

Now consider the fiber $\pi_{\rho,j}^{-1}(v,x) \subset
\FC_{\rho}$.  Recall that $\FC_{\rho}$ can be identified with a
conjugacy class of parabolic subalgebras of $\gg$, by associating to a partial flag $\{W_i\}$
its stabilizer subalgebra in $\gg$.

\begin{prop}
The intersection of $\pi_{\rho,j}^{-1}(v,x)$ with any $P$-orbit
on $\FC_\rho$ is smooth.
\end{prop}

This statement is a minor generalization of Lemma 4.3 in \cite{AH} (where only the fibers of $\pi_{\mu;\nu}$ are considered).  As in \textit{loc. cit.}, we follow the strategy of \cite[Prop 3.2]{DLP}.

\begin{proof}
Let $\{W_i\} \in \pi_{\rho,j}^{-1}(v,x)$ be a partial flag
corresponding to a parabolic subalgebra $\qg \subset \gg$.  Let $\OC$
be the $P$-orbit in $\FC_{\rho}$ of $\{W_i\}$ (or equivalently
$\qg$).  Let $Q$ be the parabolic subgroup of $G$ with Lie algebra
$\qg$.  Then the stabilizer of $\qg$ in $P$ is the intersection $H=P
\cap Q$ and $\OC = P \cdot \qg \cong P/H$.  For $p \in P$, $p\qg$ is in the fiber
$\pi_{\rho,j}^{-1}(v,x)$ if and only if $(p^{-1}v,Ad(p^{-1})x) \in
W_{j} \times \ug_Q$.

Thus the intersection $\pi_{\rho,j}^{-1}(v,x) \cap \OC$ is a
subvariety of $P/H$ of the type in \cite[Sect. 2.1]{DLP} relative to the
prehomogeneous space $\overline{P \cdot (v,x)}=V^+ \times
\ug_P$ for $P$ and the $H$-stable subspace $U = (W_{j} \times \ug_Q) \cap (V^+ \times \ug_P)$.  We
conclude that it is smooth.
\end{proof}

Recall that a finite partition of a variety $X$ into subsets is called an $\alpha$-partition if the subsets can be ordered $X_1, X_2, \ldots, X_t$ such that $X_1 \cup X_2 \cup \ldots \cup X_k$ is closed in $X$ for all $k = 1, \dots, t$.  As the Bia{\l}ynicki-Birula decomposition of $\FC_{\rho}$ with respect to $\lambda$ is an $\alpha$-partition, it follows that the intersections $\pi_{\rho,j}^{-1}(v,x) \cap \OC$, as $\OC$ runs over the $P$-orbits in $\FC_{\rho}$, form an $\alpha$-partition of $\pi_{\rho,j}^{-1}(v,x)$.

\subsection{}

We will now observe that it suffices to construct an affine paving of the fixed point sets  $(\pi_{\rho,j}^{-1}(v,x) \cap \OC)^\lambda$.
First note that we may regard $\OC$ as a vector bundle over $\OC^\lambda$ where $\lambda$ acts linearly on the fibers with strictly positive weights and $\pi_{\rho,j}^{-1}(v,x) \cap \OC \subset \OC$ is a $\mathbb G_m$-stable smooth closed subvariety.

Suppose, more generally, that $\rho: E \to Y$ is a vector bundle over a smooth variety $Y$, with a fiber preserving $\mathbb G_m$-action on $E$ with strictly positive weights and that $Z \subset E$ is a $\mathbb G_m$-stable smooth closed subvariety.

As noted in \cite[1.5]{DLP}, if $\mathbb F = \mathbb C$, one can conclude that $\pi(Z)=Z^{\GM_m}$ is smooth and $Z$ is a subbundle of $E$ restricted to $Z^{\GM_m}$.  Thus the preimage of an affine paving of $Z^{\GM_m}$ is an affine paving of $Z$.

For arbitrary characteristic, it is not clear that $Z$ must be a subbundle of $E$ over $Z^{\GM_m}$.  Nonetheless, the following result can be gleaned from~\cite[Sect. 11]{Jan}:

\begin{thm}
\label{Jan}
Let $\rho:E \to Y$ and $Z \subset E$ be as above.  Then:
\begin{enumerate}
\item the fixed point variety $Z^{\mathbb G_m}$ is smooth, and
\item if $Z^{\mathbb G_m}$ admits an affine paving, then so does $Z$.
\end{enumerate}
\end{thm}

Part (1) follows from a general result~\cite[Prop. 1.3]{iv} which states that the fixed point set of a linearly reductive group\footnote{Meaning a reductive group whose category of finite-dimensional representations is semisimple (e.g, a torus).} acting on a smooth variety is smooth.  Part (2) is a slight generalization of \cite[Lem. 11.16(b)]{Jan}, which refers to the special case when $E$ is a parabolic orbit on the full flag variety, but the proof only uses the conditions above.

We conclude that $(\pi_{\rho,j}^{-1}(v,x) \cap \OC)^\lambda$ is a smooth variety and also projective (because it is the intersection of the projective varieties $\pi_{\rho,j}^{-1}(v,x)$ and $\OC^\lambda$) and that if $(\pi_{\rho,j}^{-1}(v,x) \cap \OC)^\lambda$ admits an affine paving, then so does $\pi_{\rho,j}^{-1}(v,x)$.

\subsection{}
\label{subsec-decomp}

By the previous paragraph, it suffices to construct an affine paving of the $\lambda$-fixed point set $(\pi_{\rho,j}^{-1}(v,x) \cap \OC)^\lambda$.  We proceed by induction on the dimension of $V$.  Assume the statement is true for any vector space of dimension less than $n$.

Suppose that there is a nontrivial direct sum decomposition $V = V_1 \oplus V_2$ such that 
\begin{enumerate}
\item $V_1$ and $V_2$ are preserved by the action of $x$,
\item $V_1$ and $V_2$ are preserved by the action of the cocharacter $\lambda$, and
\item $v \in V_1 \subset V_1 \oplus V_2$.  
\end{enumerate}

Let $x_1 = x|_{V_1}$ and $x_2 = x|_{V_2}$. 

Let $\chi: \GM_m \to G$ be the cocharacter that acts on $V_1$ by scaling and on $V_2$ by the inverse.  Let $\LM = GL(V_1) \times GL(V_2)$ be the corresponding Levi subgroup and $\PM$ the corresponding parabolic.

Let $\FC_{\rho}^\chi$ be the $\chi$-fixed point set of $\FC_{\rho}$.  Each component of $\FC_{\rho}^\chi$ is contained in a unique $\PM$-orbit $\OM$ on $\FC_{\rho}$ and is in fact equal to $\OM^\chi$.  Fix $\qg \in \OM^\chi$ and let $Q\subset G$ be the corresponding parabolic subgroup and $\{ W_i\}_{i=1}^m$ the corresponding partial flag.  Then there is an isomorphism $\OM^\chi \cong \LM/ \LM \cap Q \cong \FC_{\rho'} \times \FC_{\rho''}$.  Here $\FC_{\rho'}$ and $\FC_{\rho''}$ are partial flag varieties for $GL(V_1)$ and $GL(V_2)$ respectively and $\rho'$ and $\rho''$ are sequences $0 = r'_0 < r'_1 < \dots < r'_{m'} = \dim V_1$, $0 =r''_0 < r''_1 < \dots < r''_{m''} = \dim V_1$.

The isomorphism $\OM^\chi \to  \FC_{\rho'} \times \FC_{\rho''}$ restricts to an isomorphism 
\[\pi^{-1}_{\rho,j}(v,x)^\chi \cap \OM \to \pi^{-1}_{\rho',j'}(v, x_1) \times \pi^{-1}_{\rho'',0}(0, x_2), \]
where $j'$ is defined as the number between $1$ and $m'$ such that $r'_{j'} = \dim (W_j \cap V_1)$.

This isomorphism is compatible with the action of $\lambda$, so taking $\lambda$-fixed points we obtain an isomorphism:
\[\pi^{-1}_{\rho,j}(v,x)^{\chi,\lambda} \cap \OM \to \pi^{-1}_{\rho',j'}(v,x_1)^\lambda \times \pi^{-1}_{\rho'',0}(0,x_2)^\lambda. \]
But $\pi^{-1}_{\rho,j}(v,x)^{\chi,\lambda}$ is also the $\chi$-fixed points of $\pi^{-1}_{\rho,j}(v,x)^\lambda$.  We have seen that the latter is smooth and projective, thus the Bia{\l}ynicki-Birula decomposition of $\pi^{-1}_{\rho,j}(v,x)^\lambda$ with respect to the action of $\chi$ gives an $\alpha$-partition whose pieces are locally trivial fibrations with fibers isomorphic to affine spaces over the products $\pi^{-1}_{\rho',j'}(v,x_1)^\lambda \times \pi^{-1}_{\rho'',0}(0,x_2)^\lambda$.  Applying Theorem~\ref{Jan}(2) to the $\chi$-stable smooth closed subvarieties $\pi^{-1}_{\rho,j}(v,x)^\lambda \cap \OM \subset \OM$, we find that $\pi^{-1}_{\rho,j}(v,x)^\lambda$ admits an affine paving if each $\pi^{-1}_{\rho',j'}(v, x_1)$ and $\pi^{-1}_{\rho'',0}(0, x_2)$ admit affine pavings.  The later admit affine pavings by our induction hypothesis.

\subsection{}  We call a pair $(v,x) \in \OC_{(\alpha,\beta)} \in V\times \NC$ \emph{distinguished} if for any direct sum $V=V_1 \oplus V_2$ satisfying conditions (1)-(3) of section~\ref{subsec-decomp}, either $V_1$ or $V_2$ is trivial.  By the previous paragraph, we are reduced to studying $\pi^{-1}_{\rho,j} (v,x)^\lambda$ for distinguished pairs $(v,x)$.

We first classify distinguished pairs.

\begin{lem}
If $(v,x) \in \OC_{(\alpha;\beta)}$ is distinguished then either (1) $\alpha = \emptyset$ (i.e., $v=0$) and $\beta = (n)$ (i.e., $x$ is a regular nilpotent) or (2) $\alpha = (\alpha_1, \dots, \alpha_k), \beta = (\beta_1, \dots, \beta_k)$ and $\alpha_1 > \alpha_2 > \dots > \alpha_k >0$, $\beta_1 > \beta_2 \dots > \beta_k$.
\end{lem}

\begin{proof}
Assume $(v,x) \in \OC_{(\alpha;\beta)}$ is distinguished.  For a partition $\mu$, let $\ell(\mu)$ denote the number of nonzero terms.

Suppose that $\ell(\beta) > \ell(\alpha)$, so $\beta_{\ell(\alpha)+1} > 0$.  Let $V_2 \subset V$ be the subspace spanned by the basis vectors $v_{\ell(\alpha)+1,j}$ for all $j$ and $V_1 \subset V$ be the subspace spanned by the complementary set of basis vectors.  It is clear that this is a direct sum decomposition and satisfies conditions (1)-(3) of~\ref{subsec-decomp}.  As $(v,x)$ is distinguished and $V_2$ is non-trivial by definition, we conclude that $V_1$ is trivial and so $\alpha = \emptyset$ and $\ell(\beta)=1$.

On the other hand, suppose that $\ell(\beta) \leq \ell(\alpha)$ and let $k = \ell(\alpha)$.

If $\alpha_l = \alpha_{l+1}$ for some $l < k$, we let $V_2 \subset V$ be the subspace spanned by the basis vectors $v_{l,j}$ for all $j$.  Let $V_1$ be the span of the basis vectors $v_{i,j}$ for all $i \neq l, l+1$ and the vectors $v_{l,j} + v_{l+1,j}$ for all  $j$ such that $1 \leq j \leq \alpha_{l+1} + \beta_{l+1}$.  Note that $V = V_1 \oplus V_2$, $V_1$ and $V_2$ are both nontrivial and the conditions (1)-(3) of~\ref{subsec-decomp} are satisfied.  This contradicts the assumption that $(v,x)$ be distinguished.

Similarly, suppose that $\beta_l = \beta_{l+1}$ for some $l < k$.  Let $V_1 \subset V$ be the span of the basis vectors $v_{i,j}$ for all $i \neq l, l+1$ and the vectors $x^m(v_{l,\alpha_l + \beta_l}+v_{l+1,\alpha_{l+1} + \beta_l})$ for all $m$.  Let $V_2 \subset V$ be the span of the basis vectors $v_{l+1,j}$ for all $j$.  Again we have $V = V_1 \oplus V_2$, $V_1$ and $V_2$ are both nontrivial, and the conditions (1)-(3) of~\ref{subsec-decomp} are satisfied.  This contradicts the assumption that $(v,x)$ be distinguished.
\end{proof}

We can now check that we have an affine paving in both cases.

Case (1): In this case $(\pi_{\rho,j}^{-1}(v,x) \cap \OC)^\lambda$ is either empty or a single point.

Case (2): 
As no two parts of $\alpha$ are equal, the kernel of $x$ breaks up under the action of $\lambda$ into a direct sum of 1-dimensional weight spaces with distinct weights.

For any partial flag $\{V_i\}_{i=0}^m \in \pi^{-1}_{\rho,j} (v,x)^\lambda$, $V_1$ must be contained in the kernel of $x$ and also be a direct sum of $\lambda$-weight spaces.  Let $A$ denote the finite set of such $r_1$-dimensional subspaces of the kernel of $x$.  Consider the forgetful map from $\pi^{-1}_{\rho,j} (v,x)^\lambda$ to $A$.  The fiber of this map over a point $W \in A$ is simply $\pi^{-1}_{\bar\rho,\bar{j}} (\bar{v},\bar{x})^\lambda$, where $\bar\rho = (0 < r_2-r_1 < r_3 - r_1 < \dots < r_m - r_1 = n -r_1$, $\bar{j}=j-1$ (or $0$ if $j=0$), $\bar{v}$ is the image of $v$ in the quotient $V/W$ and $\bar{x}$ is the induced action on $V/W$.  Having reduced to the case of a smaller dimensional vector space, we are done.

\bibliographystyle{myalpha}
\bibliography{gen}

\begin{thebibliography}{DCLP88}

\bibitem[AH08]{AH}
P.~N. Achar and A.~Henderson.
\newblock Orbit closures in the enhanced nilpotent cone.
\newblock {\em Adv. Math.}, 219(1):27--62, 2008.

\bibitem[AH11]{AHcor}
P.~N. Achar and A.~Henderson.
\newblock Corrigendum to ``{O}rbit closures in the enhanced nilpotent cone''
  [{A}dv. {M}ath. 219 (1) (2008) 27--62].
\newblock {\em Adv. Math.}, 228(5):2984--2988, 2011.

\bibitem[DCLP88]{DLP}
C.~De~Concini, G.~Lusztig, and C.~Procesi.
\newblock Homology of the zero-set of a nilpotent vector field on a flag
  manifold.
\newblock {\em J. Amer. Math. Soc.}, 1(1):15--34, 1988.

\bibitem[Ive72]{iv}
B.~Iversen.
\newblock A fixed point formula for action of tori on algebraic varieties.
\newblock {\em Invent. Math.}, 16:229--236, 1972.

\bibitem[Jan04]{Jan}
J.~C. Jantzen.
\newblock Nilpotent orbits in representation theory.
\newblock In {\em Lie theory}, volume 228 of {\em Progr. Math.}, pages 1--211.
  Birkh\"auser Boston, Boston, MA, 2004.

\bibitem[JMW14]{JMW2}
D.~Juteau, C.~Mautner, and G.~Williamson.
\newblock Parity sheaves.
\newblock {\em J. Amer. Math. Soc.}, 27(4):1169--1212, 2014.

\bibitem[Sun11]{Sun}
M.~Sun.
\newblock Point stabilisers for the enhanced and exotic nilpotent cones.
\newblock {\em J. Group Theory}, 14(6):825--839, 2011.

\bibitem[Tra09]{Tr}
R.~Travkin.
\newblock Mirabolic {R}obinson-{S}chensted-{K}nuth correspondence.
\newblock {\em Selecta Math. (N.S.)}, 14(3-4):727--758, 2009.

\end{thebibliography}

\end{document}